\documentclass[12pt,reqno]{amsart}

\usepackage{times}
\usepackage[T1]{fontenc}
\usepackage{mathrsfs}
\usepackage{latexsym}
\usepackage[dvips]{graphics}
\usepackage{epsfig}
\usepackage{verbatim}

\usepackage{amsmath,amsfonts,amsthm,amssymb,amscd}
\input amssym.def
\input amssym.tex
\usepackage{color}
\usepackage{hyperref}

\usepackage{color}
\usepackage{breakurl}
\usepackage{comment}
\newcommand{\bburl}[1]{\textcolor{blue}{\url{#1}}}

\newcommand{\be}{\begin{equation}}
\newcommand{\ee}{\end{equation}}
\newcommand{\bea}{\begin{eqnarray}}
\newcommand{\eea}{\end{eqnarray}}

\newtheorem{thm}{Theorem}[section]
\newtheorem{conj}[thm]{Conjecture}
\newtheorem{cor}[thm]{Corollary}
\newtheorem{lem}[thm]{Lemma}
\newtheorem{prop}[thm]{Proposition}

\newtheorem{que}[thm]{Question}

\numberwithin{equation}{section}

\usepackage{lastpage}
\usepackage{fancyhdr}
\begin{document}

\title{Domination in Direct Products of Complete Graphs}

\author{Harish Vemuri}
\address{Department of Mathematics, Yale University, New Haven, CT 06520}
\email{\textcolor{blue}{\href{mailto:harish.vemuri@yale.edu}{harish.vemuri@yale.edu}}}

\maketitle
\date{\today}

\begin{abstract} 
Let $X_{n}$ denote the unitary Cayley graph of $\mathbb{Z}/n\mathbb{Z}$. We continue the study of cases in which the inequality $\gamma_t(X_n) \le g(n)$ is strict, where $\gamma_t$ denotes the total domination number, and $g$ is the arithmetic function known as Jacobsthal's function. The best that is currently known in this direction is a construction of Burcroff which gives a family of $n$ with arbitrarily many prime factors that satisfy $\gamma_t(X_n) \le g(n)-2$. We present a new interpretation of the problem which allows us to use recent results on the computation of Jacobsthal's function to construct $n$ with arbitrarily many prime factors that satisfy $\gamma_t(X_n) \le g(n)-16$. We also present new lower bounds on the domination numbers of direct products of complete graphs, which in turn allow us to derive new asymptotic lower bounds on $\gamma(X_n)$, where $\gamma$ denotes the domination number. Finally, resolving a question of Defant and Iyer, we completely classify all graphs $G = \prod_{i=1}^t K_{n_i}$ satisfying $\gamma(G) = t+2$. 
\end{abstract}

\section{Introduction} 
Let $R$ be a finite, commutative ring with $1$. The unitary Cayley graph of $R$, denoted $X_R$, is defined as follows: the vertices of $X_R$ are the elements of $R$, and $x$ is adjacent to $y$ if and only if $x-y$ is a unit in $R$. In this paper, we will consider the graph $X_{\mathbb{Z}/n\mathbb{Z}}$, which we shall denote by $X_n$ for convenience. The motivation for studying $X_n$ comes from the theory of graph representation - in particular every finite simple graph $G$ is isomorphic to an induced subgraph of $X_n$ for some $n$. For more information on graph representation, see Gallian's dynamic survey of graph coloring \cite{G}. Properties of $X_n$ are well-studied. For instance, in 2007, Klotz and Sander \cite{KS} determined the chromatic number, independence number, clique number, and diameter of $X_n$. Various other properties were considered in \cite{A, F}. 

It is very natural to view unitary Cayley graphs as direct products of complete graphs. If $G$ and $H$ are graphs, then the \textit{direct product} of $G$ and $H$, denoted $G \times H$, is defined as follows: $V(G \times H)$ is the Cartesian product $V(G) \times V(H)$, and $(u_1, v_1)$ is adjacent to $(u_2, v_2)$ if and only if $u_1$ is adjacent to $u_2$ and $v_1$ is adjacent to $v_2$. For instance, if $n = q_1^{\alpha_1}q_2^{\alpha_2} \cdots q_k^{\alpha_k}$ where $q_1, q_2, \dots, q_k$ are distinct primes and $\alpha_1, \alpha_2, \dots, \alpha_k$ are nonnegative integers, it is straightforward to see by the Chinese Remainder Theorem that \[X_n \cong X_{q_1^{\alpha_1}} \times X_{q_2^{\alpha_2}} \times \cdots \times X_{q_k^{\alpha_k}}.\] In particular, when $\alpha_1=\alpha_2= \cdots = \alpha_k = 1$, note that $X_{n} \cong K_{q_1} \times \cdots \times K_{q_k}$. 

In this paper, we focus on two well-studied graph parameters regarding dominating sets. For a vertex $u$ of $G$, we let $N(u)$ denote the \textit{open neighborhood} of $u$, namely the set of vertices $v$ of $G$ that are adjacent to $u$. Similarly, we let $N[u]$ denote the \textit{closed neighborhood} of $u$, namely the set of vertices $v$ of $G$ that are adjacent to $u$, along with $u$ itself. Finally, for any set $S \subseteq V(G)$, we define $N(S) := \bigcup_{u \in S} N(u)$ and $N[S] = \bigcup_{u \in S} N[u]$. A \textit{dominating set} of $G$ is a set $D \subseteq V(G)$ such that $N[D] = V(G)$, and a \textit{total dominating set} of $G$ is a set $D \subseteq V(G)$ such that $N(D) = V(G)$. The \textit{domination number} of $G$, denoted by $\gamma(G)$, is the minimum cardinality of a dominating set of $G$. Similarly, the \textit{total domination number} of $G$, denoted by $\gamma_t(G)$, is the minimum cardinality of a total dominating set of $G$. It is immediately clear that $\gamma(G) \le \gamma_t(G)$ for all graphs $G$, as every total dominating set is also a dominating set. The domination numbers and total domination numbers of graphs, and particularly of graph products, are of wide interest; for more information see \cite{BKR, HL, M}.

In 2010, Meki{\u s} \cite{M} proved the following lower bound on the domination number of a direct product of complete graphs: 

\begin{thm}[\cite{M}, Theorem 2.1] \label{t+1-lwb}
Let $G = \prod_{i=1}^t K_{n_i}$ where $2 \le n_1 \le n_2 \le \cdots \le n_t$. If $t=2$, then \[\gamma(G) = \begin{cases} 2, & \text{ if } n_1 = 2;\\ 3, & \text{ if } n_1 \ge 3.\end{cases}\] If $t=3$, then $\gamma(G)=4$, and if $t \ge 4$, then $\gamma(G) \ge t+1$ with equality if $n_1 \ge t+1$. 
\end{thm}
In 2018, Defant and Iyer \cite{DI} improve this bound, proving the following result: 

\begin{thm}[\cite{DI}, Theorem 2.6]\label{t+1+-lwb}
Let $G = \prod_{i=1}^t K_{n_i}$ where $2 \le n_1 \le n_2 \le \cdots \le n_t$. If $ t \ge 4$ and $n_2 \ge 3$, then \[\gamma(G) \ge t+1 + \left \lfloor\frac{t-1}{n_1-1} \right \rfloor.\]
\end{thm}
It is worthwhile to mention that the additional assumption that $n_2 \ge 3$ loses no generality due to the following straightforward proposition from \cite{DI}.
\begin{prop}[\cite{DI}, Lemma 2.2]
Let $G = (\prod_{i=1}^s K_2) \times H$, where $s$ is a positive integer and $H$ is a finite simple graph. We have \[\gamma(G) = 2^{s-1} \gamma(K_2 \times H).\]
\end{prop}
As such, we shall be implicitly invoking this proposition by assuming that $n_2 \ge 3$ in many of our results.

A very natural application of these results is to lower bound $\gamma(X_n)$ when $n=q_1q_2 \cdots q_k$ with $q_1,q_2, \dots, q_k$ being distinct primes, wherein $X_n$ is a direct product of complete graphs. Yet, the bound in Theorem \ref{t+1+-lwb} is far from being tight when many of the $n_i$ are small and $t$ is large. To resolve this, we extend Defant and Iyer's methods to prove the following theorem: 

\begin{thm} \label{asymptotic-thm}
Take $t$ to be sufficiently large, and let $G = \prod_{i=1}^t K_{n_i}$, where $2 \le n_1 \le n_2 \le \cdots \le n_t$ and $n_2 \ge 3$. Then for any positive integer $k<(1-\frac{2}{\log_{2e}(t)})t$, we have \[\gamma(G) \ge  (t - k)\left(1+\frac{1}{n_1-1}\right) \cdots \left(1+\frac{1}{n_{k}-1}\right).\] 
\end{thm}This theorem allows us to derive new asymptotic lower bounds on $\gamma(X_n)$, where $n=q_1q_2\cdots q_t$ is a product of distinct primes. In particular, for any $\epsilon>0$, for sufficiently large $t$, we have \[\gamma(X_n) \ge (1-\epsilon)t \cdot  \prod_{i=1}^t \frac{q_i}{q_i-1}.\] 

Considering direct product graphs that do not (necessarily) arise as unitary Cayley graphs, Defant and Iyer \cite{DI} completely classify when $G = \prod_{i=1}^t K_{n_i}$ satisfies $\gamma(G) = t+1$. In particular, it is clear that for $t \ge 4$, $\gamma(G) =t+1$ if and only if $n_1 \ge t+1$ by combining Theorem \ref{t+1-lwb} and Theorem \ref{t+1+-lwb}. In their paper, they pose the question of completely classifying when $\gamma(G) = t+2$. They provide partial progress, classifying when $\gamma(G) = t+2$ under the assumption that $n_3 \ge t+1$. Furthermore their methods reduce the remaining case of $n_3 \le t$ to  the problem of classifying when $\gamma(G) = t+2$ given that $n_1 = n_2 = n_3 = t$. We resolve this problem, thus completely classifying when $\gamma(G) = t+2$ with the following theorem:
\begin{thm} \label{complete-class}
Let $G = \prod_{i=1}^t K_{n_i}$ where $2 \le n_1 \le n_2 \le \cdots \le n_t, t \ge 4$ and $n_2 \ge 3$. Then $\gamma(G) = t+2$ if and only if one of the following holds: \begin{enumerate}
    \item $n_1 = t,$ and $ n_3 \ge t+1$
    \item $\frac{t+1}{2} < n_1 \le t-1,$ and $ t+1 < n_2$
    \item $n_1 = n_2 = n_3 = t,$ and $ n_4 \ge t+2$.
\end{enumerate}
\end{thm}

Let $g(n)$ be the smallest positive integer $m$ such that any set of $m$ consecutive positive integers contains an element that is coprime to $n$. This arithmetic function is called \textit{Jacobsthal's function} and has been very well-studied due to its intimate connection with prime gaps. See \cite{MP, P} for more information. It is easy to see that $\gamma_t(X_n) \le g(n)$ as $\{0, 1, \dots, g(n)-1\}$ forms a total dominating set of $X_n$, as first stated in 2013 by Maheswari and Manjuri \cite{MM}. However, examples where this inequality is strict were unknown until the recent work of Defant and Iyer \cite{DI}, who  prove the following result:

\begin{thm}[\cite{DI}, Theorem 3.4]
Let $q \equiv 1 \pmod{3}$ be prime, let $k = \frac{2(q-1)}{3}$, and let $q_1, q_2, \dots,$ $q_k \ge q+3$ be primes. Next, let $n = 3q \prod_{i=1}^k q_i$. Then $\gamma_t(X_n) \le g(n)-1$. Consequently, there exist $n$ with arbitrarily many prime factors that satisfy the inequality $\gamma_t(X_n) \le g(n)-1$. 
\end{thm}
Defant and Iyer state that "it would be interesting to know whether $g(n)-\gamma(X_n)$ can be arbitrarily large; we have no evidence to either support or refute the claim that this is the case". In \cite{B}, Burcroff extends their argument to prove the following result: \begin{thm}[\cite{B}, Theorem 3.1] \label{gap2-const}
Let $q \equiv 1 \pmod{3}$ be prime, let $k=\frac{2(q-1)}{3}$, and let $q_1, q_2, \dots,$ $q_k \ge 2q+10$ be primes. Next, let $n=6q \prod_{i=1}^k q_i$. Then $\gamma_t(X_n) \le g(n)-2$. Consequently, there exist $n$ with arbitrarily many prime factors that satisfy the inequality $\gamma_t(X_n) \le g(n)-2$. 
\end{thm} Burcroff then poses the natural follow-up question: \begin{que} [\cite{B}]\label{gap3-que}
Does there exist a single integer $n$ that satisfies $\gamma_t(X_n) \le g(n)-3$?
\end{que}  In this paper we provide a reinterpretation of the problem that allows us to abstract away the total domination number entirely, thus allowing us to use new computational results on the Jacobsthal function to answer Burcroff's question and extend it much further. In particular, we prove the following theorem regarding the number of integers $n$ that satisfy $\gamma_t(X_n) \le g(n)-j$ for arbitrary $j \ge 0$: 

\begin{thm}\label{size-of-M-thm}
Let $j\ge 0$ be an integer. If there exists a positive integer $s$ such that $\gamma_t(X_s) \le g(s)-j$, then there exist $n$ with arbitrarily many prime factors that satisfy $\gamma_t(X_n) \le g(n)-j$.
\end{thm}
We then use this theorem along with recent computational results of Ziller \cite{Z} on Jacobsthal's function to prove the following theorem, which answers Question \ref{gap3-que} affirmatively.
\begin{thm} \label{gap-16}
There exist $n$ with arbitrarily many prime factors such that $\gamma_t(X_n) \le g(n)-16.$
\end{thm}
This large improvement on the status quo leads us naturally to the following conjecture which is stronger than Defant and Iyer's original question (note that we shall phrase this conjecture differently in Section \ref{jacobsthal}, in accordance with our reinterpretation of the problem): 
\begin{conj} \label{unbounded-gap}
For all positive integers $j$, there exist $n$ with arbitrarily many prime factors such that $\gamma_t(X_n) \le g(n)-j$. 
\end{conj}
The organization of this paper is as follows: in Section \ref{asymptotics}, we prove Theorem \ref{asymptotic-thm} and discuss how it can be used to establish asymptotic lower bounds on $\gamma(X_n)$. In Section \ref{edge-case}, we prove Theorem \ref{complete-class}. In Section \ref{jacobsthal}, we state some recent computational results on Jacobsthal's function, describe our reinterpretation of Question \ref{gap3-que}, and prove Theorems \ref{size-of-M-thm} and \ref{gap-16}. 

\subsection{Notation}
Let $G = \prod_{I=1}^t K_{n_i}$. We shall identify the vertices of $K_{n_i}$ with the elements of $\mathbb{Z}/n_i \mathbb{Z}$ for convenience. Note that we shall still be viewing $V(K_{n_i})$ and $V(K_{n_j})$ as disjoint despite the fact that we may identify some of their vertices with the same number. Now every vertex of $G$ can be identified as a $t$-tuple in which the $i$th coordinate is a vertex of $K_{n_i}$. Let $x$ be a vertex of $G$. We let $[x]_i$ denote the $i$th coordinate of $x$ when it is viewed as a $t$-tuple. Hence $[x]_i$ is a vertex of $K_{n_i}$. 

\section{Bounds on Domination in Direct Products of Complete Graphs} \label{asymptotics}
In order to improve upon the existing lower bounds on the domination numbers of direct products of complete graphs, we require the following technical lemma from \cite{DI}: 
\begin{lem}[\cite{DI}, Lemma 2.5] \label{tech}
Let $G = \prod_{i=1}^{t} K_{n_i}$, where $2 \le n_1 \le n_2 \le \cdots \le n_t$, $t \ge 4$, and $n_2 \ge 3$. Let $D$ be a dominating set of $G$ of minimum size and assume there exist disjoint, nonempty $E_1, E_2, \dots, E_k$ and a corresponding set of (distinct) indices $\{i_1, \dots, i_k\} \subseteq \{1,2, \dots, t\}$ such that for each $j \in \{1,2, \dots, k\}$, all elements of $E_j$ agree on the $i_j$th coordinate. Then letting $E = \bigcup_{j=1}^k E_j$, we have \[|E| = \sum_{j=1}^k |E_j| \le \gamma(G)-t+k+2.\] Moreover, if $|E| \ge \gamma(G)-t+k+1$, we have \[n_i \in \{2,3\} ~ \text{ for every } ~ i \not \in \{i_1, i_2, \dots, i_k\},\] and if $|E|=\gamma(G)-t+k+2$, we have $n_1=2, E=D$, and \[\{i_1,i_2, \dots, i_k\} = \{1,2, \dots, t\} \setminus \{1, h\}\] for some $h \in \{2,3, \dots, t\}$ with $n_h=3$. 

\end{lem}

As stated in \cite{DI}, an intuitive way to think about the statement of this lemma is as follows: if $D$ is a dominating set of minimum size, and we can find disjoint sets $E_1, E_2, \dots, E_k \subseteq D$ with a corresponding set of indices $\{i_1, i_2, \dots, i_k\} \subseteq \{1, 2, \dots, t\}$ such that all the elements in $E_j$ agree on the $i_j$th coordinate, then the size of $D$ cannot be too small relative to $|\bigcup_{i=1}^k E_i|$. In particular, we know that $|\bigcup_{i=1}^k E_i| \le |D|-t+k+2$, and if $|\bigcup_{i=1}^k E_i| \ge |D|-t+k+1$, we can deduce extra restrictions on the $n_i$ as well as the $i_j$.

A very rough outline of our proof of Theorem \ref{asymptotic-thm} is as follows: we greedily pick $E_i$ by invoking the pigeonhole principle on each coordinate individually, and then Lemma \ref{tech} allows us to lower bound the domination number of $G$ after performing some more involved analysis on the elements in the dominating set.

\begin{proof}[Proof of Theorem \ref{asymptotic-thm}]
Suppose for the sake of contradiction that for some $t$ sufficiently large, we have \[\gamma(G) \le (t - k)\left(1+\frac{1}{n_1-1} \right) \cdots \left(1+\frac{1}{n_{k}-1} \right).\] Rearranging this gives us \[\gamma(G)\left(1-\left(1-\frac{1}{n_1}\right)\cdots \left(1-\frac{1}{n_{k}} \right) \right) \ge \gamma(G) - t+k.\] Now let $D$ be a dominating set of $G$ with cardinality $\gamma(G)$, and for each vertex $v$ of $K_{n_i}$ let $F_v(i)$ denote the set of vertices in $D$ whose $i$th coordinate is $v$. By the pigeonhole principle, we know that there exists a vertex $v_1 \in V(K_{n_1})$ such that \[|F_{v_1}(1)| \ge \frac{\gamma(G)}{n_1},\] and similarly, for each $2 \le i \le k$, there exists a vertex $v_i \in V(K_{n_i})$ satisfying \[\left|F_{v_i}(i) \setminus \bigcup_{j=1}^{i-1}F_{v_j}(j) \right| \ge \frac{\gamma(G) - \sum_{j=1}^{i-1}|F_{v_j}(j)|}{n_i}.\] It follows from a simple induction argument that \[\left|\bigcup_{i=1}^{k}F_{v_i}(i) \right| \ge \gamma(G)\left(1-\left(1-\frac{1}{n_1}\right)\cdots \left(1-\frac{1}{n_{k}}\right) \right) \ge \gamma(G)-t+k. \] Now by hypothesis, we know that $k<(1-\frac{2}{\log_{2e}(t)})t<t-2$ for all sufficiently large $t$. Then applying Lemma \ref{tech} with $E_j = F_{v_j}(j)$ and $i_j = j$ for all $1 \le j \le k$, we know that we cannot have $|\bigcup_{i=1}^{k}F_{v_i}(i)|= \gamma(G) - t +k +2$. Thus we must have $|\bigcup_{i=1}^{k}F_{v_i}(i)| \le \gamma(G) - t +k +1$. Now from our assumption that $n_1 \ge 2,n_2 \ge 3$, we know that \[\gamma(G) \le (t-k_t)\left(1+\frac{1}{n_1-1} \right)\cdots \left(1+\frac{1}{n_{k}-1}\right) \le 2(t-k) \cdot \left(\frac{3}{2} \right)^{k-1}.\] Next, choose distinct vertices $d_1, d_2, \dots,d_{k}$ with $d_i\in F_{v_i}(i)$ for all $i$.  Furthermore, take $d_{k+1}, \dots, d_{t-1} \in D \setminus \bigcup_{i=1}^{k}F_{v_i}(i)$ to be distinct vertices. Finally, if $|\bigcup_{i=1}^{k}F_{v_i}(i)| = \gamma(G)-t+k$, take $d_t\in D \setminus \bigcup_{i=1}^{k}F_{v_i}(i)$ to be  distinct from all the other $d_i$. Otherwise, we know that $|\bigcup_{i=1}^{k}F_{v_i}(i)| = \gamma(G)-t+k + 1$, and we may take $d_t \in \bigcup_{i=1}^{k}F_{v_i}(i)$ to be distinct from all the other $d_i$.  \\
\\
Let $S_{t-1}(k)$ denote the set of elements of the symmetric group $S_{t-1}$ that fix the elements $1, 2, \dots, k$. Consider the map $f:S_{t-1}(k) \rightarrow F_{v_1}(1)$, defined by $\sigma \mapsto x_{\sigma}$, where $[x_{\sigma}]_i = [d_{\sigma(i)}]_i$ for all $i \in \{1,2, \dots, t-1\}$ and $[x_{\sigma}]_t = [d_t]_t$. It is clear that $f(\sigma) \in F_{v_1}(1)$ as we have that $[x_{\sigma}]_1 = [d_{\sigma(1)}]_1=[d_1]_1$. Now observe that \begin{align*}\frac{|S_{t-1}(k)|}{|F_{v_1}(1)|} &\ge \frac{(t-1-k)!}{2(t-k) \cdot (\frac{3}{2})^{k-1}} \\ &>\frac{(t-1-k)!}{2^{t-2} \cdot (\frac{3}{2})^{k}} \\ & > \frac{(\frac{t-k}{e})^{t-k}}{2^{t-2} \cdot e^k} \\ & > \frac{(t-k)^{t-k}}{2^{t-2} \cdot e^t} \\ &> \frac{(2t/\log_{2e}(t))^{2t/\log_{2e}(t)}}{2^{t-2} \cdot  e^t} \\ &> 4,\end{align*} where the last inequality follows upon clearing denominators, taking logs of both sides, and taking $t$ to be sufficiently large. \\
\\
This in turn implies that there exist distinct $\sigma_1, \sigma_2, \sigma_3, \sigma_4 \in S_{t-1}(k)$ such that $x_{\sigma_1}=x_{\sigma_2} = x_{\sigma_3} = x_{\sigma_4}$. It is easy to see that one of the following must hold: \begin{itemize}
    \item There is some $q \in \{k+1, \dots, t-1\}$ such that $\sigma_1(q), \sigma_2(q), \sigma_3(q), \sigma_4(q)$ are all distinct.
    \item There are distinct $q, q' \in \{k + 1, \dots, t-1\}$ such that $\sigma_1(q) \neq \sigma_2(q)$ and $\sigma_3(q') \neq \sigma_4(q')$. 
\end{itemize} Suppose the first case holds. We can apply Lemma \ref{tech} with $E_j = F_{v_j}(j)$ for $1 \le j \le k$,  $E_{k+1}= \{d_{\sigma_1(q)}, d_{\sigma_2(q)}, d_{\sigma_3(q)}, d_{\sigma_4(q)}\}$ and with $i_j = j$ for $1 \le j \le k$, and $i_{k+1}=q$. Since $|\bigcup_{i=1}^{k+1}E_i| \ge \gamma(G)-t+k + 1 + 3$, we get an immediate contradiction. \\
\\
Now suppose that the second case holds. Again we can apply Lemma \ref{tech} with $E_j = F_{v_j}(j)$ for $1 \le j \le k$, $E_{k+1} = \{d_{\sigma_1(q)}, d_{\sigma_2(q)}\}, E_{k+2} = \{d_{\sigma_3(q')}, d_{\sigma_4(q')}\}$ and with $i_j = j$ for $1 \le j \le k$, $i_{k+1}=q,$ and $i_{k+2}=q'$. Since $|\bigcup_{i=1}^{k+2}E_i| = \gamma(G)-t+k+2+2$, we know that $\{1,2, \dots, t\} \setminus \{i_1, \dots, i_{k+2}\}=\{1,h\}$ for some $h \in \{2, 3, \dots, t\}$, a contradiction as $i_1=1$. Hence our original assumption that \[\gamma(G) \le (t - k_t)\left(1+\frac{1}{n_1-1} \right) \cdots \left(1+\frac{1}{n_{k_t}-1}\right)\] must be false, as desired.
\end{proof}

We now use this result to prove new asymptotic lower bounds on $\gamma(X_n)$.
\begin{cor} \label{asymptotic-lwb}
Let $\epsilon>0$. For sufficiently large $t$, if $n = q_1 q_2 \cdots q_t$ where $q_1,q_2, \dots, q_t$ are distinct primes, then we have \[\gamma(X_n) \ge (1-\epsilon)t \cdot \prod_{i=1}^t \frac{q_i}{q_i-1}.\]
\end{cor}
\begin{proof}
Take $0<\delta<\epsilon$ and let $k = \lfloor \delta t\rfloor $. Then applying Theorem \ref{asymptotic-thm} on $G=X_n \cong \prod_{i=1}^t K_{q_i}$, we see that \begin{align*} \gamma(X_n) & \ge (1-\delta)t\cdot  \prod_{i=1}^k \frac{q_i}{q_i-1} \\ & \ge (1-\delta)t \cdot \prod_{i=k+1}^t \left(1-\frac{1}{q_i} \right) \prod_{i=1}^t \frac{q_i}{q_i-1}\\ &\ge (1-\delta)t \cdot \prod_{i=k+1}^t \left(1- \frac{1}{\delta t\log(\delta t)} \right) \prod_{i=1}^t \frac{q_i}{q_i-1} \\ & \ge (1-\delta)t \cdot  \left(1-\frac{1}{\delta t\log(\delta t)} \right)^{(1-\delta)t} \prod_{i=1}^t \frac{q_i}{q_i-1}\\ & \ge (1-\epsilon)t \cdot  \prod_{i=1}^t \frac{q_i}{q_i-1}.\end{align*} Note that the third inequality is due to the fact that for $k+1 \le i \le t$, we know $q_i \ge q_k \ge \delta t\log( \delta t)$ for sufficiently large $t$, by the Prime Number Theorem (see e.g., \cite{IK}). The final inequality holds for sufficiently large $t$ due to the fact that \[\lim_{t \rightarrow \infty} \left(1-\frac{1}{\delta t\log(\delta t)} \right)^{(1-\delta)t}=1,\] which is straightforward to confirm. 

\end{proof}

Note that one can obtain a naive lower bound on $\gamma(X_n)$, where $n=q_1q_2 \cdots q_t$ is a product of distinct primes, by arguing as follows. Every vertex in $X_n$ has degree $\prod_{i=1}^t (q_i-1)$, so for any dominating set $D$ of size $\gamma(G)$, we know that \[\prod_{i=1}^t q_i=|N[D]| \le  \gamma(G)(1+\prod_{i=1}^t(q_i-1)). \] Hence we get that asymptotically, $\gamma(G) \gtrsim \prod_{i=1}^t \frac{q_i}{q_i-1}$. We remark that the lower bound obtained in Corollary \ref{asymptotic-lwb} is significantly sharper than both the naive lower bound and the lower bound in Theorem \ref{t+1+-lwb}. 

Probabilistic arguments allow one to obtain upper bounds on $\gamma(G)$ (see, e.g., \cite{AS}). We state the result as a proposition here for convenience. 

\begin{prop}[\cite{AS}, Theorem 1.2.2] \label{prob-upperbound}
Let $G=(V,E)$ be a graph with minimum degree $\delta$ and $|V|=n$. Then \[\gamma(G) \le \frac{n(1+\log(1 +\delta))}{1+\delta}.\]
\end{prop}In particular, when $G = \prod_{i=1}^t K_{n_i}$, this proposition yields the following upper bound on $\gamma(G)$: \[\gamma(G) \le (1 + \sum_{i=1}^t \log(n_i)) \prod_{i=1}^t \left(1+\frac{1}{n_i-1}\right).\]
Combining this upper bound with the lower bound from Corollary \ref{asymptotic-lwb} allows us to obtain asymptotic estimates on $\gamma(X_n)$ where $n=p_1p_2 \cdots p_t$ is the product of the first $t$ primes. This is of particular interest due to the fact that  $\gamma(X_n)$ and $g(n)$ are closely related (a relationship that we shall consider in Section \ref{jacobsthal}), as well as the fact that $g(n)$ has been very well studied when $n=p_1p_2 \cdots p_t$ is the product of the first $t$ primes (see e.g., \cite{I} and references therein). 
\begin{cor}
Let $n=p_1p_2 \cdots p_t$ where $p_1,p_2, \dots, p_t$ are the first $t$ primes. There exist absolute constants $C_1,C_2>0$ such that for all sufficiently large $t$, \[ C_1 t \log t \le \gamma(X_n) \le C_2 t \log ^2 t .\]
\end{cor}
\begin{proof}
The upper bound follows from Proposition \ref{prob-upperbound}, the known upper bound (see e.g., \cite{IK}) \[\log(p_1p_2 \cdots p_t)< 2p_t < 4t\log(t),\] and Mertens' Third Theorem (again, see e.g., \cite{IK}). The lower bound follows from Corollary \ref{asymptotic-lwb} and Mertens' Third Theorem. 
\end{proof}
\section{Classifying when $\gamma(G) = t+2$} \label{edge-case}
In \cite{DI}, Defant and Iyer partially classify the graphs $G = \prod_{i=1}^t K_{n_i}$ such that $\gamma(G)=t+2$ with the following theorem: 
\begin{thm}[\cite{DI}, Theorem 2.8] \label{partial-class}
Let $G = \prod_{i=1}^t K_{n_i}$ where $2 \le n_1 \le n_2 \le \cdots \le n_t$, $t \ge 4$, and $n_2 \ge 3$. If $n_3 \ge t+1$, then $\gamma(G) = t+2$ if and only if one of the following holds: \begin{enumerate}
    \item $n_1=t$
    \item $\frac{t+1}{2} < n_1 \le t-1$ and $t+1 <n_2$.
\end{enumerate}
\end{thm}
Furthermore, an immediate consequence of their methods is the following proposition, which we shall also require: 
\begin{prop}[\cite{DI}] \label{reduction}
Let $G = \prod_{i=1}^t K_{n_i}$ where $2 \le n_1 \le n_2 \le \cdots \le n_t$, $t \ge 4$ and $n_2 \ge 3$. Furthermore, suppose that $n_3 \le t$ and $\gamma(G) = t+2$. Then it must be the case that $n_1 = n_2 = n_3 =t$. 
\end{prop}
Finally, we state an important lemma from \cite{DI} that we shall use frequently. 
\begin{lem}[\cite{DI}, Lemma 2.7] \label{pigeonhole-lemma}
Let $G=\prod_{i=1}^t K_{n_i}$, where $2 \le n_1 \le n_2 \le \cdots \le n_t$, $t \ge 4$ and $n_2 \ge 3$. Suppose $\gamma(G) = t+2$, and let $D$ be a dominating set of $G$ with $|D|=t+2$. For every $\ell \in \{1, 2, \dots, t\}$, and every vertex $v$ of $K_{n_{\ell}}$, let $F_{v}(\ell) = \{z \in D : [z]_{\ell} = v\}$. We have $|F_v(\ell)| \le 2$ for every choice of $\ell$ and $v$. If $\ell, \ell' \in \{1, \dots, t\}$ are distinct and there exist vertices $v$ of $K_{n_{\ell}}$ and $v'$ of $K_{n_{\ell'}}$ such that $|F_v(\ell)| = |F_v'(\ell')| = 2$, then $F_{v}(\ell) \cap F_{v'}(\ell') \neq \emptyset$.
\end{lem}

The notation introduced in Lemma \ref{pigeonhole-lemma} shall be used throughout this section. We are now ready to classify all $G = \prod_{i=1}^t K_{n_i}$ with $n_1=n_2=n_3=t$ and $\gamma(G) = t+2$.
\begin{thm} \label{final-case}
Let $G = \prod_{i=1}^t K_{n_i}$, where $2 \le n_1 \le n_2 \le \cdots \le n_t$, $t \ge 4$, and $n_1 = n_2 = n_3 = t$. Then $\gamma(G) = t+2$ if and only if $n_4 \ge t+2$.
\end{thm}
\begin{proof}
We begin by assuming that $n_4 \ge t+2$, and we shall prove that $\gamma(G)=t+2$. To see this, first note that by Theorem \ref{t+1+-lwb}, we know that $\gamma(G) \ge t+2$, so it suffices to construct a dominating set $D$ of $G$ with $|D|=t+2$. The construction is as follows - let $D = \{d_0, d_1, \dots, d_{t+1}\}$ where we define the $d_i$ as such: \[d_0:=(0,0,0,0, \dots, 0), d_1:=(0, 1, 1, 1, \dots, 1), d_2:=(1, 0, 1, 2, \dots , 2),\] \[d_3:=(1,1,0,3, \dots, 3), d_4:=(2, 2, 2, 4, \dots, 4), d_5:=(3, 3, 3, 5, \dots, 5),\] \[\cdots, d_{t+1}:=(t-1, t-1, t-1, t+1, \dots, t+1).\] 
In order to confirm that this is indeed a dominating set of $G$, we first introduce some notation. For any vertex $v \in V(G)$ and set $S \subseteq \{1,2, \dots, t\}$, let $r_S(v)$ denote the restriction of $v$ onto $S$, namely the ordered tuple \[([x]_{i_1}, [x]_{i_2}, \dots, [x]_{i_{|S|}}),\] where $i_1<i_2< \cdots < i_{|S|} \in S$. This ordered tuple corresponds to a vertex in the graph $G_S:= \prod_{i \in S} K_{n_i}$. As such, we shall use notions of adjacency and domination between restrictions of vertices onto $S$. For instance, $r_S(v_1)$ is adjacent to $r_S(v_2)$ if and only if $[v_1]_{i_j} \neq [v_2]_{i_j}$ for all $i_j \in S$. Similarly, a set of vertices $X \subseteq V(G_S)$ dominates $r_S(v)$ if and only if $r_S(v) \in X$ or some vertex $u \in X$ is adjacent to $r_S(v)$. Finally, note that if we let $T = \{1,2, \dots, t\} \setminus S$, then $v_1, v_2 \in V(G)$ are adjacent if and only if $r_S(v_1), r_S(v_2)$ are adjacent and $r_{T}(v_1), r_{T}(v_2)$ are adjacent. \\
\\
Now suppose that some vertex $x \in V(G)$ is not dominated by $D$. Let $S_1 = \{4, \dots, t\}$ and $S_2 = \{1, 2, 3\}$. Since $|S_1| = t-3$, we see that there are at least five vertices $d_{k_1}, \dots, d_{k_5} \in D$ such that $r_{S_1}(d_{k_i})$ is adjacent to $r_{S_1}(x)$ for each $1 \le i \le 5$. Now note that it suffices to show that among any five vertices $d_{k_1}, \dots, d_{k_5} \in D$, there exists some $1 \le j \le 5$ such that $r_{S_2}(d_{k_j})$ is adjacent to $r_{S_2}(x)$. To see this, we perform casework on the number of the $k_i$ $(1 \le i \le 5)$ that are less than $4$. For convenience, we shall let $R$ denote the set $\{r_{S_2}(d_{k_1}), \dots, r_{S_2}(d_{k_5})\}$. \begin{itemize}
    \item \textit{Case 1: Exactly four of the $k_i$ are less than $4$.}\\
    In this case, note that we have \[R = \{(0,0,0), (0,1,1), (1,0,1), (1,1,0), (m,m,m)\}\] for some $2 \le m \le t-1$. Now in the proof of Theorem \ref{t+1-lwb}, Meki{\u s} shows that the vertices corresponding to $(0,0,0), (0,1,1), (1,0,1), $ and $(1,1,0)$ form a dominating set for the direct product of any three complete graphs, so since $r_{S_2}(x) \in V(K_t \times K_t \times K_t)$, we know that $r_{S_2}(x)$ is dominated by one of these four vertices and hence is adjacent to some element of $R$, as desired. 
    
    \item \textit{Case 2: Exactly three of the $k_i$ are less than $4$.}\\
    In this case, note that we have \[R \supseteq \{(p,p,p), (q,q,q)\}\] for some $2 \le p < q \le t-1$. Now suppose for the sake of contradiction that $r_{S_2}(x)$ is not adjacent to an element of $R$. Notice that any element of $R$ is adjacent to either $(p,p,p)$ or $(q,q,q)$, so $r_{S_2}(x) \not \in R$. In this case, we know that one of the coordinates of $r_{S_2}(x)$ must be equal to $p$ and one must be equal to $q$. We shall only consider the case where $[r_{S_2}(x)]_1=p$ and $[r_{S_2}(x)]_2 = q$, as the other cases follow similarly. Observe that $r_{S_2}(x)$ is adjacent to any element of $\{(0,0,0), (0,1,1), (1,0,1), (1,1,0)\}$ that differs from it on its third coordinate. But since three of the $k_i$ are less than $4$, we know that there is indeed some $k_i$ such that $r_{S_2}(k_i)$ is an element of $\{(0,0,0), (0,1,1), (1,0,1), (1,1,0)\}$ that differs from $r_{S_2}(x)$ on its third coordinate, yielding the desired contradiction. Hence $r_{S_2}(x)$ is indeed totally dominated by an element of $R$.
    
    \item \textit{Case 3: At most two of the $k_i$ are less than $4$.} \\
    In this case, note that we have \[R \supseteq \{(p,p,p), (q,q,q), (r,r,r)\}\] for some $2 \le p<q<r \le t-1$. Now again, suppose for the sake of contradiction that $r_{S_2}(x)$ is not adjacent an element of $R$. Since any element of $R$ is adjacent to either $(p,p,p)$ or $(q,q,q)$, $r_{S_2}(x) \not \in R$. Thus we know that one of the coordinates of $r_{S_2}(x)$ must be equal to $p$, one must be equal to $q$, and one must be equal to $r$. Hence $r_{S_2}(x)$ is adjacent to any element of $\{(0,0,0), (0,1,1), (1,0,1), (1,1,0)\}$. Furthermore, $r_{S_2}(x)$ is adjacent to any element of the form $(u,u,u)$ where $2 \le u \le t-1$ and $u\neq p,q,r$. Since $R$ must contain some element of one of these types, we see that indeed, $r_{S_2}(x)$ is adjacent to an element of $R$, as desired. 
\end{itemize}
This shows that $D$ is a total dominating set of $G$, implying that $\gamma(G) \le t+2$ so combining this with our lower bound of $\gamma(G) \ge t+2$, we see that $\gamma(G)=t+2$, as required. \\
\\
Now we assume that $n_4 \le t+1$ and prove that $\gamma(G)>t+2$. Suppose for the sake of contradiction that $\gamma(G) = t+2$, and let $D \subseteq V(G)$ be a dominating set of size $t+2$. We shall adopt the same notation used in Lemma \ref{pigeonhole-lemma}. Note that by the pigeonhole principle along with Lemma \ref{pigeonhole-lemma}, there must be two distinct vertices $x_1,x_2 \in K_t$ such that $|F_{x_1}(1)|=|F_{x_2}(1)|=2$. Let $F_{x_1}(1)=\{u, v\}$ and $F_{x_2}(1)=\{w, z\}$, where $u,v,w,z \in D$ are all distinct. Now by an identical argument, there must be two distinct vertices $x_3,x_4 \in K_t$ such that $|F_{x_3}(2)|=|F_{x_4}(2)|=2$. By Lemma \ref{pigeonhole-lemma}, we know that \[F_{x_3}(2) \cap F_{x_1}(1) \neq \emptyset, \quad F_{x_3}(2) \cap F_{x_2}(1) \neq \emptyset,  \quad F_{x_4}(2) \cap F_{x_1}(1) \neq \emptyset, \quad F_{x_4}(2) \cap F_{x_2}(1) \neq \emptyset. \] Hence $\{F_{x_3}(2), F_{x_4}(2)\}$ is either $ \{\{u,w\}, \{v,z\}\}$  or $\{\{u, z\}, \{v, w\}\}$.\\
\\
We shall only consider the first possibility, namely $\{F_{x_3}(2), F_{x_4}(2)\} = \{\{u,w\}, \{v,z\}\}$, as the second follows identically. By the same pigeonhole argument as above, we see that there must be two distinct vertices $x_5,x_6 \in K_t$ such that $|F_{x_5}(3)|=|F_{x_5}(3)|=2$. Then by Lemma \ref{pigeonhole-lemma} we know that $F_{x_5}(3)$ has nonempty intersections with each of $\{u,v\}, \{w,z\}, \{u,w\},$ and $ \{v,z\}$, so we see that $F_{x_5}(3) \in \{\{u,z\}, \{v,w\}\}$, and identically for $F_{x_6}(3)$. Thus since $F_{x_5}(3) \neq F_{x_6}(3)$, we know that $\{F_{x_5}(3), F_{x_6}(3)\} = \{\{u,z\}, \{v,w\}\}$. Now by the pigeonhole principle, we know that there must be some vertex $x_7 \in K_{n_4}$ such that $|F_{x_7}(4)|=2$. Then by Lemma \ref{pigeonhole-lemma}, we know that $F_{x_7}(4)$ has nonempty intersections with both $\{u,v\}$ and $\{w,z\}$. Hence $F_{x_7}(4)$ is a size-$2$ subset of the set $\{u,v,w,z\}$. But note that \[\{F_{x_1}(1), F_{x_2}(1), F_{x_3}(2), F_{x_4}(2), F_{x_5}(3), F_{x_6}(3)\}\] is precisely the set of all size-$2$ subsets of $\{u, v, w, z\}$, so there is some element in this set that has empty intersection with $F_{x_7}(4)$. This contradicts Lemma \ref{pigeonhole-lemma}. Hence when $n_4 \le t+1$ we must have $\gamma(G) > t+2$, as desired. 
\end{proof}

Theorem \ref{complete-class} follows immediately upon combining Theorem \ref{partial-class}, Proposition \ref{reduction}, and Theorem \ref{final-case}. 
\section{Domination in Unitary Cayley Graphs} \label{jacobsthal} Begin by recalling from the introduction that \textit{Jacobsthal's function}, denoted by $g(n)$, is the arithmetic function defined to be the smallest positive integer $m$ such that any set of $m$ consecutive positive integers contains an element that is coprime to $n$. For $k \ge 1$, let $p_k$ denote the $k$th smallest prime. Let \[h(n)=g(p_1\cdots p_n)\quad\text{and}\quad H(n)=\max_{\omega(k)=n}g(k),\]  where $\omega(k)$ denotes the number of distinct prime factors of $k$. Furthermore, let $M_j$ denote the set of all positive integers $n$ for which $\gamma_t(X_n) \le g(n)-j$. 

A famous conjecture of Jacobsthal, which was first raised in 1962 in a letter to Erd\H{o}s \cite{E}, stated that $h(n)=H(n)$ for all positive integers $n$. Yet in 2012, Hajdu and Saradha computationally found a counterexample, thereby disproving the conjecture. 

\begin{thm}[\cite{HS}, Theorem 1.2] \label{2012-comp}
We have $H(n)=h(n)$ for all positive integers $n \le 23$. However, $H(24)=236>234 = h(24)$.
\end{thm}

In \cite{Z}, Ziller extends the results of Hajdu and Saradha, computing $H(n)$ and $h(n)$ for all $n \le 43$. Among these additional data points, Jacobsthal's conjecture fails many times. In particular, the following result is proven:
\begin{thm}[\cite{Z}] \label{2019-comp}
We have $H(41)=566$ and $h(41)=550$. 
\end{thm}
We are now ready to prove the lemma that will allow us to use these computational results to study the relationship between the total domination number of $X_n$ and $g(n)$. 
\begin{lem}\label{reinterp-lemma}
Let $G = \prod_{i=1}^tK_{n_i}$ and $H = \prod_{i=1}^t K_{m_i}$ where $t \ge 2$, $2 \le n_1 \le n_2 \le \cdots \le n_t$, $2 \le m_1 \le m_2 \le \cdots \le m_t$ and $n_i \le m_i$ for all $i \in \{1,2, \dots, t\}$. Then $\gamma_t(G) \ge \gamma_t(H)$. 
\end{lem}

\begin{proof}
Let $D$ be a total dominating set of $G$. We claim that $D$ is also a total dominating set of $H$. Suppose for the sake of contradiction that there is some element $x \in H$ that is not adjacent to any element in $D$. Let \[S = \{\ell \in \{1, 2, \dots, t\} : [x]_{\ell} \in V(K_{m_{\ell}})\setminus V(K_{n_{\ell}})\}.\] Then for each $\ell \in S$, take $v_{\ell}$ to be an arbitrary vertex of $K_{n_{\ell}}$. Now take $y \in V(G)$ such that \[[y]_j = \begin{cases} v_j & \text{ if  $j \in S$} \\ [x]_j & \text{ otherwise. }\end{cases}\] Now take some $d \in D$ that is adjacent to $y$, and note that for $j \not \in S$, we have $[d]_j \neq [x]_j$. For $j \in S$, we have $[d]_j \in V(K_{n_j})$ while $[x]_j \in V(K_{m_j}) \setminus V(K_{n_j})$ so again, we have $[d]_j \neq [x]_j$. Hence $d$ is adjacent to $x$, a contradiction. Thus $D$ is indeed a total dominating set of $H$, as desired. 
\end{proof}
It is noteworthy that this lemma allows us to abstract away the total domination number of $X_n$ from the inequality entirely and allows us to reduce the problem to just working with arithmetic properties of the Jacobsthal function. We now use Lemma \ref{reinterp-lemma} to show that $M_{16}$ is nonempty, thereby answering Burcroff's question \ref{gap3-que} in the affirmative. %With this lemma in hand, we are ready to prove that there exist $n$ with arbitrarily many prime factors that satisfy $\gamma_t(X_n) \le g(n)-16$, or equivalently, the set $\omega(M_{16})$ is unbounded. 
\begin{lem} \label{nonempty}
There exists a positive integer $n$ such that $\gamma_t(X_n) \le g(n)-16$, or equivalently, the set $M_{16}$ is nonempty. 
\end{lem}

\begin{proof}
By Theorem \ref{2019-comp}, we know that there exist $41$ distinct primes $q_1, q_2, \dots, q_{41}$ with $g(q_1q_2 \cdots q_{41}) = 566$. Moreover, if we let $p_1, p_2, \dots, p_{41}$ be the first $41$ primes, we also know that $g(p_1 p_2\cdots p_{41})=550.$ Now if we let $P=p_1p_2\cdots p_{41}$ and $Q=q_1q_2 \cdots q_{41}$, then by Lemma \ref{reinterp-lemma}, we have that \[\gamma_t(X_{Q}) \le \gamma_t(X_{P}) \le g(P)= g(Q)-16.\] Hence $Q \in M_{16}$.
\end{proof}
Before proceeding, we first state the following proposition from \cite{DI}, which is straightforward to prove. 
\begin{prop}[\cite{DI}, Lemma 2.2] \label{easyprop}
Let $n = r_1^{\alpha_1}r_2^{\alpha_2}\cdots r_k^{\alpha_k}$ where $r_1,r_2, \dots, r_k$ are distinct primes and $\alpha_1, \alpha_2, \dots, \alpha_k$ are positive integers. Furthermore, let $m=r_1r_2 \cdots r_k$. Then $\gamma_t(X_n) \ge \gamma_t(X_m)$. 
\end{prop}
This proposition yields the following useful corollary, which allows us to restrict our attention to squarefree integers. 
\begin{cor} \label{squarefree}
Let $n = r_1^{\alpha_1} r_2^{\alpha_2} \cdots r_k^{\alpha_k}$ where $r_1, r_2, \dots, r_k$ are distinct primes and $\alpha_1, \alpha_2, \dots, \alpha_k$ are nonnegative integers. Furthermore, let $m = r_1r_2 \cdots r_k$. Suppose that $n \in M_j$ for some integer $j \ge 1$. Then $m \in M_j$.
\end{cor}
\begin{proof}
First note that by definition, $g(n)=g(m)$. Thus if $\gamma_t(X_n) \le g(n)-j$, then by Proposition \ref{easyprop}, we know that \[\gamma_t(X_m) \le \gamma_t(X_n) \le g(n)-j=g(m)-j,\] as desired. 
\end{proof}
We now expand upon techniques used in \cite{B} and \cite{DI} to prove Theorem \ref{size-of-M-thm}, namely that if $M_j$ is nonempty, then $\omega(M_j)$ is unbounded.  

\begin{proof}[Proof of Theorem \ref{size-of-M-thm}]
Take $s \in M_j$ and invoking Corollary \ref{squarefree}, suppose that $s = q_1 q_2 \cdots q_m$ where $q_1, q_2, \dots, q_m$ are distinct primes. Now let $k \ge 1$ be any integer. Let $\ell = k \phi(s)$ where $\phi$ denotes Euler's totient function. Finally, take $r_1, \dots, r_{\ell}>ks + g(s)$ to be primes. We claim that $n = s \cdot \prod_{i=1}^{\ell}r_i$ is in $M_{16}$. \\
\\
Take $1 \le x \le s$ such that no integer in the set $\{x,x+1, \dots, x+g(s)-2\}$ is coprime to $s$. Note that this is possible as by the definition of $g(s)$, there must exist a sequence of $g(s)-1$ consecutive positive integers, none of which is coprime to $s$. Now let $S=\{x,x+1, \dots, x+ks+g(s)-2\}$ and let $R = \{z \in S: \gcd(z,s)=1\}$. Note that $|R| = k \phi(s)=\ell$, so let $R = \{z_1,z_2, \dots, z_{\ell}\}$.  Now by the Chinese Remainder Theorem, we know there is a unique residue $h \pmod{n}$ such that $h \equiv x \pmod{s}$ and $h \equiv -z_i \pmod{r_i}$ for all $1 \le i \le \ell$. Finally, note that $\{h, h+1, \dots, h+ks+g(s)-2\}$ is a set of $ks+g(s)-1$ consecutive integers, none of which is coprime to $n$. Consequently, we have $g(n) \ge ks+g(s)$. \\
\\
Now we shall construct a total dominating set of $X_n$ of size $ks+\gamma_t(X_s)$. First let $D=\{d_1, \dots, d_{\gamma_t(X_s)}\}$ be a total dominating set of $X_s$. For $1 \le i \le \gamma_t(X_s)$, define $x_i$ to be the unique vertex in $X_n$ such that $x_i \equiv d_i \pmod{s}$ and $x_i \equiv -1 \pmod{r_j}$ for all $1 \le j \le \ell$. We claim that $D' =\{0, 1, \dots, ks-1\} \cup \{x_1, \dots, x_{\gamma_t(X_s)}\}$ is a total dominating set of $X_n$. To confirm this, suppose for the sake of contradiction that $y \in V(X_n)$ is not adjacent to any element of $D'$. Then we know that $y$ is not adjacent to any element in $\{0, 1, \dots, ks-1\}$, so in particular, the set \[S=\{y, y-1, \dots, y-(ks-1)\}\] consists of $ks$ consecutive integers, none of which are coprime to $n$. Now for $r \in \mathbb{N}$, let $B(r) = \{h \in S: \gcd(h,r) \neq 1\}$. First note that $|B(s)|=k(s-\phi(s))$ and for each $1 \le j \le \ell$, $|B(r_j)| \le 1$. Then observe that \[S = B(s) \cup \bigcup_{j=1}^{\ell} B(r_j),\] and by a union bound, we have \begin{align*}ks &= |S| \\ &= \left| B(Q) \cup \bigcup_{j=1}^{\ell} B(r_j)\right| \\ & \le |B(Q)|+\sum_{j=1}^{\ell}|B(r_j)| \\ & \le k(s-\phi(s)) + k\phi(s) \\ &= ks. \end{align*} Thus in fact, every inequality above must be an equality, implying that $B(r_j)=1$ for all $1 \le j \le \ell$. Hence for each $1 \le j \le \ell$, there must be some corresponding $0 \le h \le ks-1$ such that $y \equiv h \pmod{r_j}$. Then since $r_j>ks+g(s)$ for all $1 \le j \le \ell$, we know that we cannot have $y \equiv -1 \pmod{r_j}$ for any $ 1 \le j \le \ell$. Then since $D$ is a total dominating set of $X_s$, we know that there must be some $1 \le h \le \gamma_t(X_s)$ such that $\gcd(y-d_h,s)=1$. Then since $y-x_h \not \equiv 0 \pmod{r_j}$ for all $1 \le j \le \ell$, we know that $\gcd(y-x_h,n)=1$, so $y$ is adjacent to $x_h$. This contradicts our original assumption. Consequently, $D'$ must be a total dominating set of $X_n$, as claimed. \\
\\
Hence we know that $\gamma_t(X_n) \le ks+\gamma_t(X_s)$ and $ks+g(s) \le g(n)$. Then since $s \in M_j$, we know that $\gamma_t(X_s) \le g(s)-j$, which in turn implies that $\gamma_t(X_n) \le ks +g(s)-j \le g(n)-j$, as desired. 
\end{proof}
Theorem \ref{gap-16} now immediately follows from Lemma \ref{nonempty} and Theorem \ref{size-of-M-thm}. Interestingly, if we apply Lemma \ref{reinterp-lemma} on Theorem \ref{2012-comp} and then invoke Theorem \ref{size-of-M-thm}, we obtain a new family of $n$ with arbitrarily many prime factors that satisfy $\gamma_t(X_n) \le g(n)-2$, distinct from the family found by Burcroff in Theorem \ref{gap2-const}. Now the computational results of Ziller \cite{Z} seem to suggest the following result, which we conjecture to be true:
\begin{conj}
For all positive integers $j$, there exists some positive integer $k$ along with primes $q_1<q_2< \cdots < q_k$ and primes $r_1<r_2< \cdots <r_k$ where $q_i \le r_i$ for all $i \in \{1,2, \dots, k\}$, that satisfy the inequality \[g(q_1q_2 \cdots q_k) \le g(r_1r_2 \cdots r_k)-j.\]
\end{conj}
If this statement were true, then an application of Lemma \ref{reinterp-lemma} along with Theorem \ref{size-of-M-thm} would imply that the set $\omega(M_j)$ is unbounded for all $j$, thus proving Conjecture \ref{unbounded-gap} and completely resolving the original question posed by Defant and Iyer \cite{DI}.
\section*{Acknowledgements}
 This research was funded by NSF / DMS grant 1659047 and NSA grant H98230-18-1-0010 as part of the 2019 Duluth Research Experience for Undergraduates (REU) program. The author would like to thank Joe Gallian for suggesting the problem and running the program. The author would also like to thank Aaron Berger, Amanda Burcroff, and Colin Defant for helpful discussions.  

\bigskip

\end{document}